\documentclass[11pt]{article}
\usepackage{amsthm, amsmath, amssymb, amsfonts, url, booktabs, tikz, setspace, fancyhdr, bm,xcolor}

\usepackage{mathrsfs} 
\usepackage[margin = 2.2cm]{geometry}
\usepackage{hyperref, enumerate}
\usepackage[shortlabels]{enumitem}
\usepackage[babel]{microtype}
\usepackage[english]{babel}
\usepackage[capitalise]{cleveref}
\usepackage{comment}
\usepackage{bbm}
\usepackage{csquotes}
\usepackage{mathabx}
\usepackage{tikz}
\usepackage{graphicx}
\usepackage{float}
\usepackage{thm-restate}
\usepackage{mathtools}
\usetikzlibrary{arrows.meta}
\usepackage{subcaption}
\usepackage{accents}

\usepackage{tikz}
\usepackage{graphicx}
\usepackage{float}
\usetikzlibrary{positioning}
\usetikzlibrary{shapes,arrows,arrows,positioning,fit}

\counterwithin{figure}{section}

\newtheorem{theorem}{Theorem}[section]
\newtheorem{prop}[theorem]{Proposition}
\newtheorem{lemma}[theorem]{Lemma}
\newtheorem{cor}[theorem]{Corollary}

\newtheorem{cons}{Construction}
\newtheorem{claim}[theorem]{Claim}
\newtheorem{case}{Case}

\newtheorem{question}{Question}

\newtheorem{property}[theorem]{Property}
  
\theoremstyle{definition}
\newtheorem{definition}[theorem]{Definition}

\theoremstyle{definition}

\newtheorem*{defn-non}{Definition}

\newenvironment{pf}[1][Proof]{%
    \begingroup
    \begin{proof}[#1]%
}{%
    \end{proof}
    \endgroup
}

\title{A directed Andr\'asfai-Erd\H{o}s-S\'os theorem and chromatic profiles of oriented cycles}
\author{Yisai Xue\thanks{School of Mathematics and Statistics, Ningbo University, Ningbo, China. Supported by the National Nature Science Foundation of China (No. 12501486). Email: xueyisai@nbu.edu.cn}}
\date{\today}

\begin{document}

\maketitle

\begin{abstract}
    The chromatic profile of a digraph $H$, denoted by $\delta_{\chi}^{+}(H,k)$, is the infimum $d$ such that any $H$-free digraph $D$ on $n$ vertices with minimum out-degree $\delta^{+}(D) \ge dn$ must be $k$-colorable. 
    We determine the exact chromatic profile for several fundamental classes of digraphs. 
    Our main result is a directed analogue of the Andr\'asfai-Erd\H{o}s-S\'os theorem, stating that $\delta_\chi^{+}(T_r, r-1)=\frac{3 r-7}{3 r-4}$, where $T_r$ is the transitive tournament on $r$ vertices. 
    We then determine the chromatic profile for directed odd cycles, showing that $\delta^+_\chi(\overrightarrow{C}_{2\ell+1},2)=1/2$ for all $\ell\ge 1$. 
    Finally, we resolve the profile for the three remaining orientations of the pentagon, establishing that $\delta_{\chi}^{+}(C_{5}',2)=\delta_{\chi}^{+}(C_{5}'',2)=\delta_{\chi}^{+}(C_{5}''',2)=1/3$.
\end{abstract}

\section{Introduction}

The celebrated Andr\'asfai-Erd\H{o}s-S\'os theorem \cite{andrasfai1974connection} states that for $r \ge 3$, any $n$-vertex $K_r$-free graph $G$ with minimum degree $\delta(G) > \frac{3r-7}{3r-4}n$ is $(r-1)$-colorable.
This bound is sharp (see \cref{fig:AES}).
The general problem of determining the minimum degree threshold that guarantees a certain chromatic number in $H$-free graphs is known as  the chromatic profile problem.
Formally, the \textit{chromatic profile of a graph} $H$ is defined as
\begin{align*}
\delta_\chi(H, k):=\inf \{d: \delta(G) \geq d|G| \text { and } H \not\subseteq G \Rightarrow \chi(G) \leq k\}.
\end{align*}

The case of triangle-free graphs has been extensively studied. It is known from the work of Andr\'asfai, Erd\H{o}s and S\'os \cite{andrasfai1974connection}, Brandt and Thomass\'e \cite{brandt2011dense}, H\"aggkvist \cite{haggkvist1982odd} and Jin \cite{jin1995triangle} that
$$
\delta_\chi(K_3, 2)=\frac{2}{5}, \quad \delta_\chi(K_3, 3)=\frac{10}{29} \quad \text { and } \quad \delta_\chi(K_3, k)=\frac{1}{3} \quad \text { for every } k \geq 4.
$$

Goddard and Lyle \cite{goddard2011dense}, and independently Nikiforov \cite{nikiforov2010chromatic} extended these results by showing that $\delta_\chi(K_r, r)=\frac{19r-47}{19r-28}$ and $\delta_\chi(K_r, r+1)=\frac{2r-5}{2r-3}$.
For odd cycles, H\"aggkvist \cite{haggkvist1982odd} showed that $\delta_\chi(C_{2\ell+1},2)=\frac{2}{2\ell+3}$ for $1\le \ell\le 4$.
More recently, Yuan and Peng \cite{yuan2024minimum} established that $\delta_\chi(C_{2\ell+1},2)=\frac{1}{6}$ for $\ell\ge 5$.
The problem for an arbitrary graph $H$ was also posed by Erd\H{o}s and Simonovits \cite{erdos1973valence}, who described it as ``too complicated" in its full generality.
For more related results, see \cite{allen2013chromatic,bottcher2023graphs,bourneuf2025dense,huang2025interpolating,Illingworth2022-2,Illingworth2022-1,kim2025stability,liu2024andr,liu2024positive,liu2024beyond,yan2024chromatic}.

The chromatic profile problem connects extremal and coloring theory by asking how a ``local'' density condition (minimum degree), rather than a ``global'' one, constrains the chromatic number of an $H$-free graph. This framework, exemplified by the Andr\'asfai-Erd\H{o}s-S\'os theorem, studies the stability of chromatic properties in graphs that are dense but not necessarily extremal.

The extension to digraphs introduces two primary challenges: the proliferation of forbidden structures due to their various orientations, and the structural asymmetries that demand new proof techniques.

All digraphs considered in this paper are finite, may contain anti-parallel arcs (i.e., directed 2-cycles), but have no loops or parallel arcs. 
The \emph{chromatic profile of a digraph $H$} is defined as:
\begin{align}\label{eq:profile}
\delta^+_\chi(H, k):=\inf \{d: \delta^+(D) \geq d|D| \text { and } H \not\subseteq D \Rightarrow \chi(D) \leq k\}.
\end{align}
Here, $\delta^{+}(D)$ denotes the minimum out-degree of the digraph $D$. 
Our first result generalizes the Andr\'asfai-Erd\H{o}s-S\'os theorem to the directed setting by determining the profile of the transitive tournament $T_r$ on $r$ vertices.

\begin{theorem}\label{thm:1.1}
    $\delta^+_\chi(T_r,r-1)=\frac{3r-7}{3r-4}$.
\end{theorem}

Notably, by replacing each edge in an undirected graph with a pair of anti-parallel arcs, one can recover the Andr\'asfai-Erd\H{o}s-S\'os theorem as a direct corollary of \cref{thm:1.1}.
Moreover, as an application, we obtain the following stability theorem.

\begin{theorem}\label{thm:app-of-AES}
    Let $r \geq 3, t \geq 1$ be integers and $\varepsilon>0$. Let $D$ be a $T_{r}[t]$-free digraph on $n$ vertices with minimum out-degree $\delta^{+}(D) \geq(\frac{3 r-7}{3 r-4}+\varepsilon) n$. If $n$ is sufficiently large, then one can delete $o(n^2)$ arcs to make $D$ an $(r-1)$-partite digraph.
\end{theorem}

 Our second result determines the 2-coloring profile for directed odd cycles.

\begin{theorem}\label{thm:directed-cycles}
   For every $\ell\ge 1$, we have $\delta^+_\chi(\overrightarrow{C}_{2\ell+1},2)=\frac{1}{2}$.
\end{theorem}

\begin{figure}[ht]
\centering
\resizebox{0.7\linewidth}{!}{%
\begin{tikzpicture}[scale=1] 
  \coordinate (A) at (0,0); 
  \coordinate (B) at (2,0); 
  \coordinate (C) at (2.618,1.902); 
  \coordinate (D) at (1,3.078); 
  \coordinate (E) at (-0.618,1.902); 

  \draw[-latex, red, line width=1.501pt] (A) -- (E);
  \draw[-latex, red, line width=1.501pt] (E) -- (D);
  \draw[-latex, red, line width=1.501pt] (D) -- (C);
  \draw[-latex, red, line width=1.501pt] (C) -- (B);
  \draw[-latex, red, line width=1.501pt] (B) -- (A);

  \foreach \i in {A,B,C,D,E}
   \draw[fill=white, line width=0.701pt] (\i) circle (3pt);

  \coordinate (A1) at (0+5.0001,0); 
  \coordinate (B1) at (2+5.0001,0); 
  \coordinate (C1) at (2.618+5.0001,1.902); 
  \coordinate (D1) at (1+5.0001,3.078); 
  \coordinate (E1) at (-0.618+5.0001,1.902); 

  \draw[-latex, red, line width=1.501pt] (A1) -- (E1);
  \draw[-latex, red, line width=1.501pt] (E1) -- (D1);
  \draw[-latex, red, line width=1.501pt] (D1) -- (C1);
  \draw[-latex, red, line width=1.501pt] (C1) -- (B1);
  \draw[-latex, blue, line width=1.501pt] (A1) -- (B1);

  \foreach \i in {B1,C1,E1,A1,D1}
   \draw[fill=white, line width=0.701pt] (\i) circle (3pt);

  \coordinate (A2) at (0+10.0001,0); 
  \coordinate (B2) at (2+10.0001,0); 
  \coordinate (C2) at (2.618+10.0001,1.902); 
  \coordinate (D2) at (1+10.0001,3.078); 
  \coordinate (E2) at (-0.618+10.0001,1.902); 

  \draw[-latex, red, line width=1.501pt] (A2) -- (E2);
  \draw[-latex, red, line width=1.501pt] (E2) -- (D2);
  \draw[-latex, red, line width=1.501pt] (D2) -- (C2);
  \draw[-latex, blue, line width=1.501pt] (B2) -- (C2);
  \draw[-latex, blue, line width=1.501pt] (A2) -- (B2);

  \foreach \i in {A2,B2,C2,D2,E2}
   \draw[fill=white, line width=0.701pt] (\i) circle (3pt);

  \coordinate (A3) at (0+15.0001,0); 
  \coordinate (B3) at (2+15.0001,0); 
  \coordinate (C3) at (2.618+15.0001,1.902); 
  \coordinate (D3) at (1+15.0001,3.078); 
  \coordinate (E3) at (-0.618+15.0001,1.902); 

  \draw[-latex, red, line width=1.501pt] (A3) -- (E3);
  \draw[-latex, red, line width=1.501pt] (E3) -- (D3);
  \draw[-latex, blue, line width=1.501pt] (C3) -- (D3);
  \draw[-latex, red, line width=1.501pt] (C3) -- (B3);
  \draw[-latex, blue, line width=1.501pt] (A3) -- (B3);

  \foreach \i in {A3,B3,C3,D3,E3}
   \draw[fill=white, line width=0.701pt] (\i) circle (3pt);

   \coordinate (a) at (1,-0.8);
   \node[above]  at (a) {$\overrightarrow{C}_5$};
   \coordinate (b) at (6,-0.8);
       \node[above]  at (b) {$C'_5$};
   \coordinate (c) at (11,-0.8);
   \node[above]  at (c) {$C''_5$};
   \coordinate (d) at (16,-0.8);
   \node[above]  at (d) {$C'''_5$};
\end{tikzpicture}
}
\caption{Orientations of $C_5$}
\label{fig:Orientations of C5}
\end{figure}
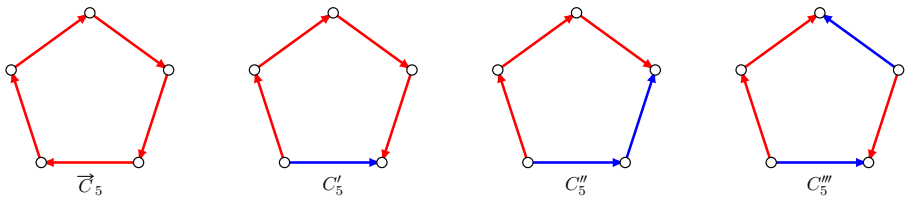

Finally, we investigate the chromatic profile for orientations of the pentagon $C_5$. There are four non-isomorphic orientations of the pentagon, depicted in \cref{fig:Orientations of C5}.
While \cref{thm:directed-cycles} addresses the profile for $\overrightarrow{C}_{5}$, our final result resolves the 2-coloring profiles for the three remaining cases.

\begin{theorem}\label{thm:C5}
    $\delta^+_\chi(C'_5,2)=\delta^+_\chi(C''_5,2)=\delta^+_\chi(C'''_5,2)=\frac{1}{3}$.
\end{theorem}

The remainder of this paper is organized as follows. In \cref{sec2}, we prove our main result, a directed analogue of the Andr\'asfai-Erd\H{o}s-S\'os theorem, which determines the chromatic profile for transitive tournaments. \cref{sec3} is devoted to directed odd cycles, where we establish their 2-coloring profile. In \cref{sec4}, we complete the picture for pentagon orientations by resolving the 2-coloring profiles for the remaining three cases. We then present an application of our main theorem in \cref{sec5}, proving a stability theorem. Finally, in \cref{sec6}, we conclude with a discussion of related concepts and pose several open questions for future research.

\medskip
\paragraph{Notations.} 
    We write $u \to v$ or $(u,v)$ to denote an arc from $u$ to $v$, and say that $u$ \textit{dominates} $v$. 
    For a vertex $v$ in a digraph $D$, its \textit{in-neighbourhood} $N_{D}^{-}(v)$ is the set $\{u \in V(D) : u \to v\}$, and its \textit{out-neighbourhood} $N_{D}^{+}(v)$ is the set $\{u \in V(D) : v \to u\}$. 
    The \textit{in-degree} $d_{D}^{-}(v)$ and \textit{out-degree} $d_{D}^{+}(v)$ are the sizes of these sets, respectively. 
    A vertex of in-degree zero is called a \textit{source}; one of out-degree zero a \textit{sink}.
    We omit the subscript $D$ when the context is clear. 
    For $v\in V(D)$ and $S\subseteq V(D)$, let $d^+(v,S)$ denote the number of arcs from $v$ to $S$.
    Similarly, let $d^+(S,v)$ denote the number of arcs from $S$ to $v$.

    For two graphs $G$ and $H$, the \emph{join} $G \vee H$ is formed by taking disjoint copies of $G$ and $H$, and adding all edges between $V(G)$ and $V(H)$.  
    The $s$-blowup of a graph $G$, denoted by $G[s]$, is obtained by replacing each vertex $v \in V(G)$ with an independent set $I_v$ of size $s$, and adding all edges between $I_u$ and $I_v$ whenever $uv \in E(G)$.
    For $S \subseteq V(D)$, $D[S]$ denotes the subgraph induced by $S$. 
    The \textit{underlying graph} $G(D)$ of a digraph $D$ is the simple graph on $V(D)$ where $\{u,v\}$ is an edge if and only if $(u,v)$ or $(v,u)$ is an arc in $D$.
    For a graph $G$, we let $\overleftrightarrow{G}$ denote its double orientation, where each edge $uv$ is replaced by the pair of arcs $(u,v)$ and $(v,u)$. 
    
    The \textit{odd girth} $g_{\text{odd}}(D)$ of a digraph $D$ is the length of a shortest odd cycle in its underlying graph $G(D)$. 
    The \textit{independence number} $\alpha(D)$ is the size of a largest independent set in $G(D)$. 
    A \textit{tournament} is an orientation of a complete graph; it is \textit{transitive} if the arc relation is transitive, i.e., $((a \rightarrow b)$ and $(b \rightarrow c)) \Rightarrow(a \rightarrow c)$.    
    A \textit{homomorphism} from a digraph $D$ to a digraph $H$ is a mapping $\phi:V(D)\rightarrow V(H)$ such that if $(u,v)$ is an arc in $D$, then $(\phi(u),\phi(v))$ is an arc in $H$, denoted by $D\xrightarrow{\textup{hom}} H$.
    We use $[k]$ to denote the set $\{1, \dots, k\}$.
    For clarity, we omit floor and ceiling signs when they are not essential to the argument.

\section{A Directed Analogue of the Andr\'asfai-Erd\H{o}s-S\'os Theorem}\label{sec2}

We begin by recalling two elementary properties of transitive tournaments that follow directly from the definition. 
\begin{property}\label{property}
Let $T_r$ be a transitive tournament on $r$ vertices.
    \begin{itemize}
        \item[\rm (i)] Every vertex-induced subgraph of $T_r$ is a transitive tournament; 
        \item[\rm (ii)] Adding a new vertex $v$ and arcs such that $v$ dominates all vertices in $V(T_r)$, or is dominated by all of them, results in a transitive tournament on $r+1$ vertices.
    \end{itemize}
\end{property}

To prove \cref{thm:1.1}, we first handle the special case where the underlying graph is complete multipartite.

\begin{lemma}\label{lmm:multipartite}
   Let $D$ be an $n$-vertex digraph whose underlying graph $G(D)$ is complete $s$-partite with $s\ge r$. If $\delta^{+}(D)>\frac{3r-7}{3r-4}n$, then $D$ contains a copy of $T_{r}$.
\end{lemma}

\begin{proof}
    We proceed by induction on $r$. For the base case $r=3$, the condition $s \ge 3$ implies that $G(D)$ contains a triangle. 
    Let its vertices be $v_1, v_2, v_3$.
    As $\delta^+(D) > \frac{2}{5}n$, the pigeonhole principle implies that some pair $\{v_i,v_j\}$ has a common out-neighbor $w$.
    Then subgraph induced by $\{v_i, v_j, w\}$ contains a $T_3$, since $v_i \to v_j$, $v_i \to w$, and $v_j \to w$.

    Now, assume the lemma holds for all $r'<r$. 
    Let the color classes of $D$ be $A_{1}, \dots, A_{s}$. 
    Suppose to the contrary that $D$ is $T_{r}$-free.
    The minimum out-degree condition implies that $\alpha(D) = \max_{i\in[s]} |A_i| < n - \frac{3r-7}{3r-4}n = \frac{3}{3r-4}n$. 
    We first establish a key property of out-neighborhoods in $D$.

\begin{claim}\label{cl:Dv}
    For every $v\in V(D)$, the subgraph $D_{v}:=D[N^{+}(v)]$ is $(r-2)$-colorable.
\end{claim}

\begin{pf}
   Clearly, $D_{v}$ is $T_{r-1}$-free; otherwise, since $v$ dominates all vertices in $N^{+}(v)$, the set $V(T_{r-1}) \cup \{v\}$ induces a copy of $T_r$ by \cref{property} (ii).
    Furthermore, we have $|D_v|=d^+(v)>\frac{3r-7}{3r-4}n$ and
\begin{align*}
    \delta^+(D_v)>\frac{3r-7}{3r-4}n-(n-|D_v|)>\frac{3(r-1)-7}{3(r-1)-4}|D_v|.
\end{align*}
    Since $D_v$ is $T_{r-1}$-free and complete multipartite, by the induction hypothesis, we have $\chi(D_v)\le r-2$. 
\end{pf}

    Moreover, we must have $\chi(D_v) = r-2$. Otherwise, $\alpha(D_v) \ge \frac{|D_v|}{r-3} > \frac{(3r-7)/(r-3)}{3r-4}n> \frac{3n}{3r-4}$, contradicting $\alpha(D_v) \le \alpha(D) < \frac{3n}{3r-4}$.
    Let $B_1,\dots,B_{r-2}$ be the color classes of $D_v$. 
    Note that $B_i$ is a subset of some $A_j$.
    Then $|B_i|< \frac{3}{3r-4}n$, and for each $i\in [r-2]$, we have 
\begin{align*}
    |B_i|=d^+(v)-\sum_{j\neq i}|B_j|> \frac{3r-7}{3r-4}n-(r-3)\frac{3n}{3r-4}=\frac{2n}{3r-4}.
\end{align*}
    This implies a crucial fact: if a vertex $u$ has a positive in-degree (i.e., is dominated by some vertex $v$), then $u$ belongs to a color class $B_i$ in the coloring of $D_v$. Since $B_i$ is a subset of some partition class $A_j$ of $D$ and $|B_i| > \frac{2n}{3r-4}$, it follows that $|A_j| > \frac{2n}{3r-4}$.

    Now, let $S = V(D) \setminus V(D_v)$. Then $|S|=n-d^+(v) < \frac{3n}{3r-4}$. 
    Since $D$ is $s$-partite with $s \geq r$ and $N^+(v)$ can be covered by $r-2$ of these color classes, at least two of the color classes of $D$ must be entirely contained in $S$.
    By the pigeonhole principle, at least one such set, say $A_i$, has size $|A_i|\le \frac{|S|}{2}<\frac{1.5n}{3r-4}<\frac{2n}{3r-4}$, which implies that every vertex in $A_i$ is a source (i.e., with zero in-degree).
    As $D$ is complete $s$-partite, any source vertex must dominate all vertices outside its own color class.
    Thus, for every $v\in A_i$, we have $N^+(v)=V(D)\setminus A_i$.
    Moreover, $\chi(D-A_i)=\chi(D_v)\le r-2$ by \cref{cl:Dv}, which implies that $s=\chi(D) \leq \chi(D - A_i) + 1 \le r-1$, contradicting $s \geq r$. 
    Hence $D$ contains $T_r$.    
\end{proof}

\begin{lemma}[Principle of Inclusion-Exclusion]\label{lmm:Inclusion-Exclusion}
    Let $A_1, \ldots, A_p$ be finite sets. Then
$$
|A_1 \cap \cdots \cap A_p| \geq \sum_{i=1}^p|A_i|-(p-1)\left|\bigcup_{i=1}^p A_i\right| .
$$
\end{lemma}

The next lemma provides a simpler, yet weaker, sufficient condition for finding a $T_{r}$.

\begin{lemma}\label{lmm:T_r}
 If an $n$-vertex digraph $D$ has $\delta^{+}(D)\ge\frac{r-2}{r-1}n+1$, then $D$ contains a copy of $T_{r}$.
\end{lemma}

\begin{proof}
    We proceed by induction on $r$.
    The base case $r=2$ is trivial.
    For the inductive step, assume $r\ge 3$ and the lemma holds for all integers less than $r$.
    Since $\sum_{x\in V(D)}d^+(x)=\sum_{x\in V(D)}d^-(x)$,
  by the pigeonhole principle, there is a vertex $u$ with $d^-(u)\ge \frac{r-2}{r-1}n+1$.
  The minimum out-degree condition also implies $d^+(u)\ge \frac{r-2}{r-1}n+1$.
  Let $S=N^{+}(u)\cap N^{-}(u)$. Then $$|S|\ge d^-(u)+d^+(u)-n\ge \frac{r-3}{r-1}n+2.$$
  Consider the induced subgraph $D[S]$. We have
 $$\delta^+(D[S])\ge \delta^+(D)-(n-|S|)\ge  \frac{r-4}{r-3}|S|+1.$$ 
 By the induction hypothesis, $D[S]$ contains $T_{r-2}$, denoted by $Q$.
  By \cref{lmm:Inclusion-Exclusion}, we have 
\begin{align*}
    \left|\bigcap_{v\in \{u\}\cup V(Q)} N^+(v)\right|\ge (r-1)\left(\frac{r-2}{r-1}n+1\right)-(r-2)n>0.
\end{align*}
  Therefore, $u$ and $V(Q)$ have a common out-neighbor $v$.
  Note that $u$ dominates $\{v\}\cup V(Q)$, and $v$ is dominated by $u$ and $V(Q)$, $u$, $v$ and $Q$ form a copy of $T_r$ by \cref{property} (ii).
\end{proof}

Brandt \cite{brandt2003structure} gave a simple proof of the Andr\'asfai–Erd\H{o}s–S\'os theorem using 5-wheel-like graphs. We extend this concept to digraphs.

\begin{definition}[5-wheel-like digraph $\overrightarrow{W}_{r,t}$; see Figure \ref{fig:5-wheel-like}]
    A 5-wheel-like digraph $\overrightarrow{W}_{r,t}$ consists of vertices $v$, $w_1$, $w_2$, and transitive tournaments $Q_1$, $Q_2$ such that:
    \begin{enumerate}
        \item $Q_1$ and $Q_2$ are transitive tournaments on $r-2$ vertices with $|V(Q_1) \cap V(Q_2)| = t$;
        \item For each $i \in [2]$, both $V(Q_i) \cup \{v\}$ and $V(Q_i) \cup \{w_i\}$ induce transitive tournaments on $r-1$ vertices;
        \item $vw_1, vw_2 \notin E(G(\overrightarrow{W}_{r,t}))$ and $w_1w_2 \in E(G(\overrightarrow{W}_{r,t}))$.
    \end{enumerate}
\end{definition}

\begin{figure}[H]
\centering
\resizebox{0.3\linewidth}{!}{%
\begin{tikzpicture}[scale=1] 
  \coordinate (A) at (0,0); 
  \coordinate (B) at (2,0); 
  \coordinate (C) at (2.618,1.902); 
  \coordinate (D) at (1,3.078); 
  \coordinate (E) at (-0.618,1.902); 

  \coordinate (A1) at (-0.5,1.55);
  \coordinate (A2) at (0.8,1.45);
  \coordinate (B1) at (2.5,1.55);
  \coordinate (B2) at (1.2,1.45);
  \coordinate (D1) at (-0.5,2.05);
  \coordinate (D2) at (2.5,2.05);
  \coordinate (D3) at (1,2.1);

  \draw[thick, red] (A) -- (B);
  \draw[thick, red] (A) -- (A1);
  \draw[thick, red] (A) -- (A2);
  \draw[thick, red] (B) -- (B1);
  \draw[thick, red] (B) -- (B2);
  \draw[thick, red] (D) -- (D1);
  \draw[thick, red] (D) -- (D2);
  \draw[thick, red] (D) -- (D3);

  \foreach \i in {A,B,D}
   \draw[fill=white, line width=0.701pt] (\i) circle (3pt);

   \draw[thick, blue] (0.3,1.8) ellipse (1 and 0.4);
   \draw[thick, blue] (1.7,1.8) ellipse (1 and 0.4);

   \coordinate (Q1) at (-0.7,1.8);
   \coordinate (Q2) at (2.7,1.8);
   \coordinate (v) at (1,3.2);
   \coordinate (w1) at (0,-0.65);
   \coordinate (w2) at (2,-0.65);
   \node[left]  at (Q1) {$Q_1$};
   \node[right]  at (Q2) {$Q_2$};
   \node[above] at (v) {$v$};
   \node[above] at (w1) {$w_1$};
   \node[above] at (w2) {$w_2$};
\end{tikzpicture}
}
\caption{The 5-wheel-like digraph $\overrightarrow{W}_{r,t}$.}
\label{fig:5-wheel-like}
\end{figure}

We are now ready to prove \cref{thm:directed-AES}. We use a saturation argument, dividing the proof into two cases based on whether the underlying graph of the saturated digraph is complete multipartite.

\begin{theorem}\label{thm:directed-AES}
   Let $D$ be a $T_r$-free digraph on $n$ vertices. If $\delta^+(D)> \dfrac{3r-7}{3r-4}n$, then $D$ is  $(r-1)$-colorable.
\end{theorem}

\begin{proof}
  Consider the saturated digraph $\hat{D}$ of $D$, which is obtained by adding arcs until any further addition creates $T_r$.
  Then $\delta^+(\hat{D})\ge \delta^+(D)> \dfrac{3r-7}{3r-4}n$.
  If $G(\hat{D})$ is complete multipartite, by \cref{lmm:multipartite}, $\hat{D}$ is $(r-1)$-colorable, so is $D$.

    Now suppose $G(\hat{D})$ is not complete multipartite.
    This implies the existence of a 5-wheel-like digraph.

\begin{claim}
    $\hat{D}$ contains a 5-wheel-like digraph $\overrightarrow{W}_{r,t}$.
\end{claim}    

\begin{pf}
   Since $G(\hat{D})$ is not complete multipartite, there exist vertices $v,w_1,w_2$ such that $vw_1, vw_2 \notin E(G(\hat{D}))$ and $w_1w_2 \in E(G(\hat{D}))$. As $\hat{D}$ is saturated, for each $i \in [2]$, there exists a transitive tournament $Q_i$ of size $r-2$ such that $Q_i$ with $v$, $w_i$, and the arc $(v,w_i)$ form $T_r$. By \cref{property}(i), $Q_i \cup \{v\}$ and $Q_i \cup \{w_i\}$ form $T_{r-1}$, yielding $\overrightarrow{W}_{r,t}$, where $t=|V(Q_1)\cap V(Q_2)|$.
\end{pf}

 Choose a copy $W$ of $\overrightarrow{W}_{r,t}$ in $\hat{D}$ with maximum $t$.
 Let $Q_1$ and $Q_2$ be the copies of $T_{r-2}$ in $\overrightarrow{W}_{r,t}$.
 First, we provide an upper bound on $t$.

\begin{claim}
    $t\le r-3$.
\end{claim}

\begin{pf}
    Suppose to the contrary that $t=r-2$. 
    Then $V(Q_1)=V(Q_2)$.
    Let $U=V(Q_1)\cup \{w_1,w_2\}$.
    Since $w_i\cup V(Q_i)$ induces a copy of $T_{r-1}$, every vertex is dominated by at most $r-1$ vertices in $U$.
    Let $Y$ be the set of vertices dominated by exactly $r-1$ vertices in $U$, and $Z = V(\hat{D}) \setminus Y$. 
    Then every vertex in $Z$ is dominated by at most $r-2$ vertices in $U$.
    Obviously, $\hat{D}[Y]$ is $T_{r-2}$-free. By \cref{lmm:T_r}, there exists $y \in Y$ with $d^+_{\hat{D}[Y]}(y)\le \frac{r-4}{r-3}|Y|$. Then
    \[
    \delta^+(\hat{D}) \le d^+(y) \le \frac{r-4}{r-3}|Y| + |Z|,
    \]
    which implies
    \begin{align}\label{eq:111}
        (r-3)\delta^+(\hat{D}) \le (r-4)(n-|Z|)+(r-3)|Z|= (r-4)n + |Z|.
    \end{align}
    Also,
    \begin{align}\label{eq:222}
        r\delta^+(\hat{D}) \le \sum_{v \in U} d^+(v) \le (r-2)|Z| + (r-1)|Y| = (r-1)n - |Z|.
    \end{align}
    Combining \eqref{eq:111} and \eqref{eq:222} yields $\delta^+(\hat{D}) \le \frac{2r-5}{2r-3}n<\frac{3r-7}{3r-4}n$, a contradiction.
\end{pf}

 Note that $w_i\cup V(Q_i)$ induces a copy of $T_{r-1}$.
 We have $d^+(W,u)\le |W|-1$ for every $u\in V(\hat{D})$.
 Define $X= \bigcap_{v\in V(Q_1)\cap V(Q_2)}N^+(v)$.
 Then by \cref{lmm:Inclusion-Exclusion}, we have 
\begin{align}\label{eq:X-lower}
    |X|\ge t\delta^+(\hat{D})-(t-1)n.
\end{align}
 
 Moreover, for each $x\in X$, there exists $y_i\in (\{w_i\}\cup V(Q_i))\setminus V(Q_{3-i})$ that does not dominate $x$.
 Hence, $d^+(W,x)\le |W|-2$.
 We will further bound the number of vertices in $W$ that dominate $x$.

\begin{claim}\label{cl:2.10}
    For every $x\in X$, we have $d^+(W,x)\le |W|-3$.
\end{claim}

\begin{pf}
    If $v$ does not dominate $x$, then $d^+(W,x)\le |W|-3$. 
    We now assume $v\to x$.
    Then there exist $y_1\in V(Q_1)\setminus V(Q_2)$ and $y_2\in V(Q_2)\setminus V(Q_1)$ that do not dominate $x$; otherwise, $\hat{D}$ contains $T_r$ by \cref{property} (ii).
    Let $Q_i'=Q_1-y_i+x$ for $i\in [2]$.
    Then $Q_1'$ and $Q_2'$ are transitive tournaments on $r-2$ vertices and $|V(Q_1')\cap V(Q_2')|=t+1$.
    Note that $v$, $w_1$, $w_2$, $Q_1'$, $Q_2'$ form a copy of $\overrightarrow{W}_{r,t+1}$, contradicting the maximality of $t$.
    Thus $d^+(W,x)\le |W|-3$ for every $x\in X$. 
\end{pf}

By \cref{cl:2.10}, we have
\begin{align}\label{eq:X-upper}
    |W|\delta^+(\hat{D})\le \sum_{w\in W}d^+(w)\le (|W|-3)|X|+(|W|-1)(n-|X|).
\end{align}

    Note that $|W|=2r-t-1$.  
    Combining \eqref{eq:X-lower} and \eqref{eq:X-upper} yields $ \delta^+(\hat{D})\le \frac{2r+t-4}{2r+t-1}n\le \frac{3r-7}{3r-4}n$, a contradiction.
    Hence, $D$ is  $(r-1)$-colorable.
\end{proof}

\begin{cons}
    Let $G_n^3=C_5[\frac{n}{5}]$ and $G_n^r=C_5[\frac{n}{3r-4}]\vee K_{r-3}[\frac{3n}{3r-4}]$.
\end{cons}

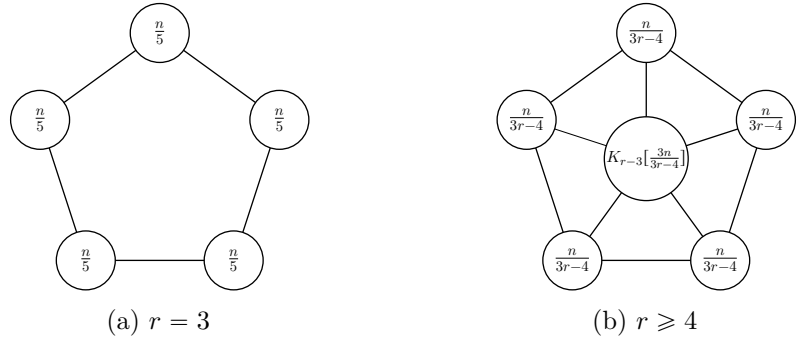
\begin{figure}[htbp]
    \centering

    \begin{subfigure}[b]{0.3\textwidth}
        \centering
    \resizebox{0.8\linewidth}{!}{
 \begin{tikzpicture}
  \def\R{3}
  \def\r{1}

  \foreach \i in {0,1,2,3,4}{
    \coordinate (P\i) at (90+\i*72:\R); 
  }

  \draw[thick, line width=0.901pt] (P0)--(P1)--(P2)--(P3)--(P4)--cycle;

  \foreach \i in {0,1,2,3,4}{
    \draw[fill=white, line width=0.901pt] (P\i) circle (20pt);
  }
  \node at (P0) {\Large $\frac{n}{5}$};
  \node at (P1) {\Large $\frac{n}{5}$};
  \node at (P2) {\Large $\frac{n}{5}$};
  \node at (P3) {\Large $\frac{n}{5}$};
  \node at (P4) {\Large $\frac{n}{5}$};
\end{tikzpicture}
}
    \caption{$r=3$}
    \end{subfigure}
    \hspace{30pt}
    \begin{subfigure}[b]{0.3\textwidth}
        \centering
    \resizebox{0.8\linewidth}{!}{
\begin{tikzpicture}
  \def\R{3}
  \def\r{1}

  \foreach \i in {0,1,2,3,4}{
    \coordinate (P\i) at (90+\i*72:\R); 
  }

  \draw[thick, line width=0.901pt] (P0)--(P1)--(P2)--(P3)--(P4)--cycle;

  \coordinate (c) at (0,0);
  \draw[thick, line width=0.901pt] (P0)--(c);
  \draw[thick, line width=0.901pt, line width=0.701pt] (P1)--(c);
  \draw[thick, line width=0.901pt] (P2)--(c);
  \draw[thick, line width=0.901pt] (P3)--(c);
  \draw[thick, line width=0.901pt] (P4)--(c);

  \foreach \i in {0,1,2,3,4}{
    \draw[fill=white, line width=0.901pt] (P\i) circle (20pt);
  }

  \draw[fill=white, line width=0.901pt] (0,0) circle (\r);

  \node at (c) {$K_{r-3}[\frac{3n}{3r-4}]$};
  \node at (P0) {\Large $\frac{n}{3r-4}$};
  \node at (P1) {\Large $\frac{n}{3r-4}$};
  \node at (P2) {\Large $\frac{n}{3r-4}$};
  \node at (P3) {\Large $\frac{n}{3r-4}$};
  \node at (P4) {\Large $\frac{n}{3r-4}$};
\end{tikzpicture}
}
        \caption{$r\ge 4$}
    \end{subfigure}
\caption{Lower bound for Andr\'asfai-Erd\H{o}s-S\'os theorem.}
\label{fig:AES}
\end{figure}

\begin{proof}[\textbf{Proof of Theorem \ref{thm:1.1}}]
  For the lower bound, consider the digraph $\overleftrightarrow{G_n^r}$ formed by the double orientation of $G_n^r$.
  Then $\overleftrightarrow{G_n^r}$ is $T_r$-free with $\chi(\overleftrightarrow{G_n^r})=r$ and $\delta^+(\overleftrightarrow{G_n^r})=\frac{3r-7}{3r-4}n$.
  This construction shows that $\delta^+_\chi(T_r,r-1)\ge\frac{3r-7}{3r-4}$.
  The upper bound, $\delta^+_\chi(T_r,r-1)\le \frac{3r-7}{3r-4}$, is precisely the statement of \cref{thm:directed-AES}.
    Combining these two bounds yields the desired equality.
\end{proof}

\section{Chromatic Profile of Directed Odd Cycles}\label{sec3}

We begin by constructing families of digraphs that provide lower bounds for the chromatic profiles.

\begin{cons}[See \cref{fig:cons} (a)]
 Let $K_{\lfloor \frac{n-1}{2} \rfloor, \lceil \frac{n-1}{2} \rceil}$ be the balanced complete bipartite graph on $n-1$ vertices and $\overleftrightarrow{K}$ be its double orientation.
 We define the digraph $A_n$ by taking the vertex set $V(\overleftrightarrow{K}) \cup \{v\}$ for some new vertex $v$, and adding all arcs from $v$ to $V(\overleftrightarrow{K})$.
\end{cons}

\begin{prop}\label{prop:3.1}
    The digraph $A_n$ is $\overrightarrow{C}_{2\ell+1}$-free for any $\ell\ge 1$ and satisfies $\chi(A_n)=3$ and $\delta^+(A_n)=\lfloor \tfrac{n}{2}\rfloor$.
\end{prop}

\begin{lemma}\label{thm:directed-odd-cycle-upper-bdd}
   Let $\ell\ge 3$ be an integer, and let $D$ be a digraph on $n$ vertices. If $\delta^{+}(D) \geq (n+\ell-2)/2$, then $D$ contains $\overrightarrow{C_\ell}$.
\end{lemma}

\begin{proof}
  Suppose that $D$ is a digraph with $\delta^+(D)\geq (n+\ell-2)/2$.
  Since $\sum_{x\in V(D)}d^+(x)=\sum_{x\in V(D)}d^-(x)$,
  by the pigeonhole principle, there is a vertex $u$ with $d^-(u)\ge (n+\ell-2)/2$.
  Hence, $|N^+(u)\cap N^-(u)|\ge |N^+(u)|+|N^-(u)|-n>0$, which implies $N^+(u)\cap N^-(u)\neq \varnothing$.
  Let $v_1\in N^+(u)\cap N^-(u)$. Consider the subgraph induced by $N^-(u)$.
  Note that each vertex in $N^-(u)$ has at least $(n+\ell-2)/2-(n-(n+\ell-2)/2)=\ell-2$ out-neighbors within $N^-(u)$.
  By iteratively selecting out-neighbors within $N^-(u)$, we construct a directed path $v_1,v_2,\ldots,v_{\ell-1}$ in $N^-(u)$.
  Since arcs $(u,v_1)$ and $(v_{\ell-1},u)$ exist, we find a directed cycle of length $\ell$.
\end{proof}

\begin{proof}[\textbf{Proof of Theorem \ref{thm:directed-cycles}}]
    \cref{prop:3.1} shows that $\delta^+_\chi(\overrightarrow{C}_{2\ell+1},2)\ge 1/2$ and \cref{thm:directed-odd-cycle-upper-bdd} implies that  $\delta^+_\chi(\overrightarrow{C}_{2\ell+1},2)\le 1/2$.
    Thus $\delta^+_\chi(\overrightarrow{C}_{2\ell+1},2)=1/2$.
\end{proof}

\section{Chromatic Profile of Other Pentagon Orientations}\label{sec4}

In contrast to the directed cycle, the three remaining pentagon orientations require a more intricate analysis. We begin by providing the constructions that establish the lower bounds for their profiles.

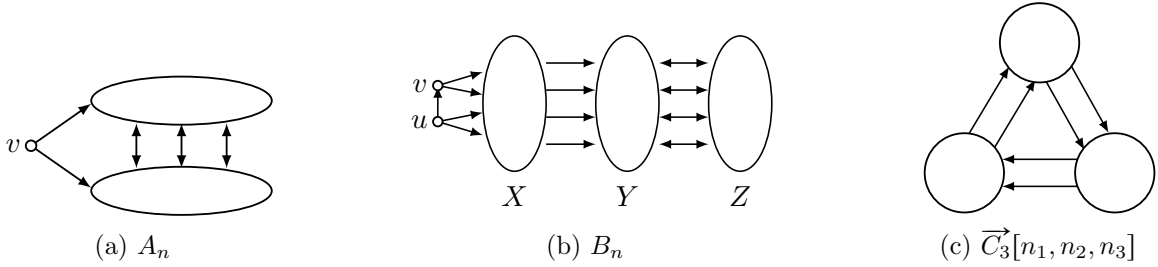
\begin{figure}[H]
    \centering

    \begin{subfigure}[b]{0.3\textwidth}
        \centering
        \begin{tikzpicture}[scale=0.4]
    
  \draw[line width=0.701pt] (5, 1.5) ellipse (3cm and 0.8cm);
  \draw[line width=0.701pt] (5, -1.5) ellipse (3cm and 0.8cm);
  \draw[latex-latex,line width=0.701pt] (5, 0.75) -- (5, -0.75);
  \draw[latex-latex,line width=0.701pt] (3.5, 0.75) -- (3.5, -0.75);
  \draw[latex-latex,line width=0.701pt] (6.5, 0.75) -- (6.5, -0.75);
  \draw[-latex,line width=0.701pt] (0, 0) -- (2, 1.4);
  \draw[-latex,line width=0.701pt] (0, 0) -- (2, -1.4);
  \draw[fill=white, line width=0.701pt] (0, 0) circle (5pt);  
  \coordinate (v) at (0, 0);
  \node[left] at (v) {$v$};
\end{tikzpicture}
        \caption{$A_n$}
    \end{subfigure}
    \hfill
    \begin{subfigure}[b]{0.3\textwidth}
        \centering
        \begin{tikzpicture}[scale=0.6]  
  \draw[line width=0.601pt] (3, 3) ellipse (0.7cm and 1.5cm);

  \draw[line width=0.601pt] (5.5, 3) ellipse (0.7cm and 1.5cm);
  \draw[-latex,line width=0.601pt] (3.7, 3.9) -- (4.8, 3.9);
  \draw[-latex,line width=0.601pt] (3.7, 3.3) -- (4.8, 3.3);
  \draw[-latex,line width=0.601pt] (3.7, 2.7) -- (4.8, 2.7);
  \draw[-latex,line width=0.601pt] (3.7, 2.1) -- (4.8, 2.1);

  \draw[line width=0.601pt] (8, 3) ellipse (0.7cm and 1.5cm);
  \draw[latex-latex,line width=0.601pt] (6.2, 3.9) -- (7.3, 3.9);
  \draw[latex-latex,line width=0.601pt] (6.2, 3.3) -- (7.3, 3.3);
  \draw[latex-latex,line width=0.601pt] (6.2, 2.7) -- (7.3, 2.7);
  \draw[latex-latex,line width=0.601pt] (6.2, 2.1) -- (7.3, 2.1);

  \draw [-latex,line width=0.601pt] (1.3, 2.6) -- (1.3, 3.4);
  \draw [-latex,line width=0.601pt] (1.3, 3.4) -- (2.3, 3.7);
  \draw [-latex,line width=0.601pt] (1.3, 3.4) -- (2.3, 3.2);
  \draw [-latex,line width=0.601pt] (1.3, 2.6) -- (2.3, 2.3);
  \draw [-latex,line width=0.601pt] (1.3, 2.6) -- (2.3, 2.8);

  \draw[fill=white, line width=0.701pt] (1.3, 2.6) circle (3pt); 
  \draw[fill=white, line width=0.701pt] (1.3, 3.4) circle (3pt); 

  \coordinate (X) at (3,0.5);
  \coordinate (Y) at (5.5,0.5);
  \coordinate (Z) at (8,0.5);
  \coordinate (u) at (1.3, 2.6);
  \coordinate (v) at (1.3, 3.4);
  \node[above] at (X) {$X$};
  \node[above] at (Y) {$Y$};
  \node[above] at (Z) {$Z$};
  \node[left] at (u) {$u$};
  \node[left] at (v) {$v$};
\end{tikzpicture}
        \caption{$B_n$}
    \end{subfigure}
    \hfill
    \begin{subfigure}[b]{0.3\textwidth}
        \centering
        \begin{tikzpicture}[scale=1] 
  \coordinate (A) at (0,0); 
  \coordinate (B) at (2,0); 
  \coordinate (C) at (1,1.73); 
  \coordinate (A1) at (-0.25,0); 
  \coordinate (A2) at (0.2,0); 
  \coordinate (A3) at (0.5,0.18); 
  \coordinate (A4) at (0.5,-0.18); 
  \coordinate (B1) at (2,0.18); 
  \coordinate (B2) at (1.6,0.3);
  \coordinate (B3) at (1.92,0.53);
  \coordinate (B4) at (2,-0.18);
  \coordinate (C1) at (0.6,1.42); 
  \coordinate (C2) at (0.95,1.25);
  \coordinate (C3) at (1.35,1.55);
  \coordinate (C4) at (0.9,1.55);

  \draw[-latex,line width=0.601pt] (A1) -- (C1);  
  \draw[-latex,line width=0.601pt] (A2) -- (C2); 
  \draw[-latex,line width=0.601pt] (C4) -- (B2); 
  \draw[-latex,line width=0.601pt] (C3) -- (B3); 
  \draw[-latex,line width=0.601pt] (B1) -- (A3); 
  \draw[-latex,line width=0.601pt] (B4) -- (A4);

  \foreach \i in {A,B,C}
   \draw[fill=white, line width=0.701pt] (\i) circle (15pt);

\end{tikzpicture}
        \caption{$\overrightarrow{C_3}[n_1, n_2, n_3]$}
    \end{subfigure}

    \caption{Illustrations of the constructions $A_n$, $B_n$, and $ \overrightarrow{C_3}[n_1, n_2, n_3]$.}
    \label{fig:cons}
\end{figure}

\begin{cons}[For $C'_5$, see \cref{fig:cons} (b)]\label{cons3}
Define $B_n$ as the digraph with:
    \begin{itemize}
        \item $V(B_n)=\{u,v\}\cup X\cup Y\cup Z$ with $|X|=\lceil\frac{n-2}{3}\rceil$, $|Y|=\lfloor\frac{n-2}{3}\rfloor$, and $|Z|=n-2-|X|-|Y|$;
        \item $A(B_n)=\{(u,v)\}\cup \{(w,x):w\in \{u,v\},x\in X\}\cup \{(x,y):x\in X, y\in Y\}\cup \{(y,z):y\in Y,z\in Z\}\cup \{(z,y):z\in Z, y\in Y\}$.
    \end{itemize}
\end{cons}

\begin{prop}\label{prop:3.2}
The digraph $B_n$ is $C'_5$-free and satisfies $\chi(B_n)=3$ and $\delta^+(B_n)=\lfloor \tfrac{n-2}{3}\rfloor$.
\end{prop}

\begin{proof}
    Clearly, $\delta^+(B_n)=\lfloor \tfrac{n-2}{3}\rfloor$.
    To show that $B_n$ is $C'_5$-free, we observe that any 5-cycle in the underlying graph $G(B_n)$ must contain the edge $uv$. 
    Such a cycle must be of the form $u-v-x_1-y-x_2-u$ for some $x_1,x_2\in X$ and $y\in Y$. 
    However, inspecting the arc directions in $B_n$, no such cycle corresponds to an orientation isomorphic to $C'_5$.
\end{proof}

\begin{cons}[For $C''_5$ and $C'''_5$, see \cref{fig:cons} (c)]\label{cons4}
   Let $\overrightarrow{C_3}[n_1, n_2, n_3]$ be the blow-up of $\overrightarrow{C_3}$, where the vertex parts have sizes $n_1=\lceil n/3 \rceil$, $n_2=\lfloor n/3 \rfloor$, and $n_3=n-n_1-n_2$. 
\end{cons}

\begin{prop}\label{prop:3.3}
   The digraph $\overrightarrow{C_3}[n_1, n_2, n_3]$ is $\{C_{5}'', C_{5}'''\}$-free with chromatic three and minimum out-degree  $\lfloor n/3 \rfloor$.
\end{prop}

\begin{proof}
  Let $D=\overrightarrow{C_3}[n_1, n_2, n_3]$.
     Clearly, $\chi(D)=3$ and $\delta^+(D)=\lfloor n/3 \rfloor$.
    Since every oriented $C_5$ in $D$ admits a homomorphism to $\overrightarrow{C_3}$, whereas neither $C''_5$ or $C'''_5$ admits a homomorphism to $\overrightarrow{C_3}$, $D$ is $\{C_{5}'', C_{5}'''\}$-free.
\end{proof}

The proof of the upper bound requires several auxiliary results. 

\begin{lemma}\label{fact}
    Let $k\ge 2$ be an integer and $\overrightarrow{P}$ an oriented path on $k$ vertices.
    Then there exists an integer $t\le k$ such that $\overrightarrow{P}\xrightarrow{\textup{hom}} \overrightarrow{P_t}$.
\end{lemma}

\begin{proof}
   We argue by induction on $k$. The case $k=2$ is trivial.
    For the inductive step, assume the statement holds for paths with fewer than $k$ vertices, where $k\ge 3$.
    Let $\overrightarrow{P}$ have vertex sequence $v_1,v_2,\dots,v_k$, and let $\overrightarrow{P'}$ denote the subpath from $v_1$ to $v_{k-1}$.
    By the inductive hypothesis, there exist $t'\le k-1$ and a homomorphism $\phi':\overrightarrow{P'}\xrightarrow{\textup{hom}} \overrightarrow{P_{t'}}$ with image vertices $u_1,u_2,\dots,u_{t'}$ in order.
    Set $t=t'+1$, and extend $\overrightarrow{P_{t'}}$ to $\overrightarrow{P_{t}}=(u_1,\ldots,u_{t'},u_{t'+1})$ and $\overrightarrow{P'_{t}}=(u_0,u_1,\ldots,u_{t'})$.
    Note $\overrightarrow{P_{t}}\cong\overrightarrow{P'_{t}}$.

    To extend $\phi'$ to $\phi$ on $v_k$, consider the last arc between $v_{k-1}$ and $v_k$:
\begin{align*}
\phi(v_i)=
\begin{cases}
\phi'(v_i), & 1\le i\le k-1,\\
u_{p+1}, & i=k,~\phi'(v_{k-1})=u_p,\text{ and }v_{k-1}\to v_k,\\
u_{p-1}, & i=k,~\phi'(v_{k-1})=u_p,\text{ and }v_{k}\to v_{k-1}.
\end{cases}
\end{align*}
The orientation of the last arc is preserved by mapping $v_k$ to the predecessor/successor of $u_p$ accordingly.
Hence $\phi:V(\overrightarrow{P_{k}})\to V(\overrightarrow{P_{t}})$ is a homomorphism with $t\le k$.
\end{proof}

\begin{lemma}[Gao, Liu, Wu and Xue, \cite{rainbow}]\label{prop:intersection}
    Let $t\in \mathbb{N}$ and $\varepsilon\in (0,1)$. Then there exist $\alpha =\alpha(\varepsilon,t)$ and $m=m(\varepsilon,t)$ such that for any $V_1, V_2, \cdots, V_m\subseteq [n]$ each with size at least $\varepsilon n$, there exist $1\le i_1< i_2<\dots < i_t\le m$ with
    $|\bigcap_{j=1}^{t} V_{i_j}| \ge \alpha n. $
\end{lemma}

Let $D$ be a digraph with minimum out-degree of linear size $n$.  
\cref{prop:intersection} implies that among any $m$ vertices of $D$, there are $k$ vertices that share a common out-neighborhood of linear size $n$.
Iterating the above proposition yields the following corollary.

\begin{cor}\label{cor:path-hom}
  Let $D$ be a digraph on $n$ vertices with $\delta^+(D)\ge \varepsilon n$.
    For any $k$ and $\varepsilon>0$, there exists an $\alpha':=\alpha'(\varepsilon,k)$ such that, there exist disjoint vertex sets $X_1,\ldots,X_k\subseteq V(D)$ such that $|X_i|\ge\alpha' n$ for all $i\in[k]$ and 
\begin{align*}
    X_{i+1}\subseteq \bigcap_{v\in X_{i}} N^+(v),\quad \forall~ i\in [k-1].
\end{align*}
\end{cor}

\begin{lemma}\label{cor:cycles}
Let $0<\varepsilon \le 1/2$ and $k>1$ be an integer.
    Suppose $D$ is a digraph on $n$ vertices with $\delta^{+}(D) \geq (1/2+\varepsilon)n$, where $n$ is sufficiently large.
   Then $D$ contains every orientation of $C_k$.
\end{lemma}

\begin{proof}
   By \cref{thm:directed-odd-cycle-upper-bdd}, $D$ contains a directed $k$-cycle.
   Suppose $C$ is an oriented $k$-cycle that is not isomorphic to $\overrightarrow{C_k}$.
   It is a known property that any oriented cycle that is not a directed cycle must contain a sink. We can thus choose $v$ to be such a sink vertex.
   Then $P=C-v$ is an oriented path of length $k-1$.
   We first show that there is an oriented path $P'$ in $D$ with start vertex $u$ and end vertex $w$ that is isomorphic to $P$.
   By \cref{fact}, it suffices to show that there exists a $k$ blow-up of $\overrightarrow{P_t}$, where $\overrightarrow{P'}\xrightarrow{\textup{hom}} \overrightarrow{P_t}$.
   By \cref{cor:path-hom}, there exists an $\alpha':=\alpha'(t)$ such that, there exist disjoint vertex sets $X_1,\ldots,X_t\subseteq V(D)$ such that $|X_i|\ge \alpha'n> k$ for all $i\in[t]$ and 
\begin{align*}
    X_{i+1}\subseteq \bigcap_{v\in X_{i}} N^+(v),\quad \forall~ i\in [t-1].
\end{align*}
    We can now embed the path $P'$ into $D$ by choosing its vertices greedily. Since the sets $X_i$ are large and fully connected to the next set in the sequence, we can always pick a vertex for the path that satisfies the required adjacencies and has not been picked before.
   Since $\delta^{+}(D) \geq (1/2+\varepsilon)n$, there is a vertex $v'\in (N^+(u)\cap N^+(w))\setminus V(P')$.
   Then $P'$ and $v'$ form a copy of $C$.
\end{proof}

The core of our argument is the following lemma, which guarantees the existence of all three pentagon orientations in any sufficiently dense digraph that contains a triangle.

\begin{lemma}\label{lmm:og=3}
    Let $\varepsilon>0$ and $D$ be a digraph with $\delta^+(D)\ge (1/3+\varepsilon)|D|$ and $g_{\text{odd}}(D)=3$.
    Then $D$ contains $C_5'$, $C''_5$ and $C'''_5$.
\end{lemma}

\begin{proof}
Suppose $D$ is an $n$-vertex digraph with $\delta^+(D)\ge (1/3+\varepsilon)n$ and $g_{\text{odd}}(D)=3$.
We prove the existence of each orientation in turn.

\noindent\textbf{(i) Existence of $C''_5$.}

Let $\{u, v, w\}$ be the vertices of a triangle in $G(D)$. 
A standard averaging argument on the out-neighborhoods of these three vertices shows that at least one pair, say $\{u,v\}$, must have at least $\varepsilon n$ common out-neighbors. 
Assume without loss of generality that $u\rightarrow v$. 
Select any three distinct vertices $x_1, x_2, x_3 \in N^+(u) \cap N^+(v)$. 
Again by averaging, two of these, say $x_1$ and $x_3$, share at least $\varepsilon n$ common out-neighbors. 
Choose $y\in N^+(x_1) \cap N^+(x_3)\setminus \{u,v\}$, then the vertices $\{u, v, x_1, x_3, y\}$ form a copy of $C_5''$, with directed paths $u \to v \to x_1 \to y$ and $u \to x_3 \to y$ (see \cref{fig:C''}).

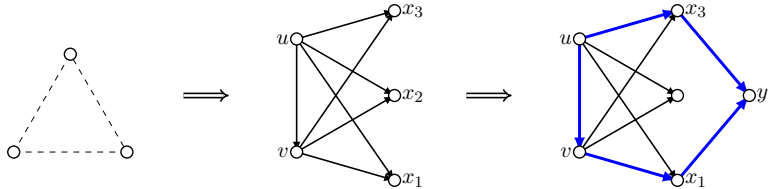
\begin{figure}[H]
\centering
\resizebox{0.6\linewidth}{!}{%
\begin{tikzpicture}[scale=1] 
  \coordinate (A) at (0,0); 
  \coordinate (B) at (2,0); 
  \coordinate (C) at (1,1.73); 

  \coordinate (A1) at (5,0);
  \coordinate (B1) at (5,2);
  \coordinate (C1) at (6.73,1);
  \coordinate (C2) at (6.73,1+1.5);
  \coordinate (C3) at (6.73,1-1.5);

  \draw[dashed] (A) -- (B) -- (C)  -- cycle;  

  \draw[latex-, thick] (A1) -- (B1);
  \draw[-latex, thick] (A1) -- (C1);
  \draw[-latex, thick] (A1) -- (C2);
  \draw[-latex, thick] (A1) -- (C3);
  \draw[-latex, thick] (B1) -- (C1);
  \draw[-latex, thick] (B1) -- (C2);
  \draw[-latex, thick] (B1) -- (C3);

  \coordinate (D1) at (5+5,0);
  \coordinate (E1) at (5+5,2);
  \coordinate (F1) at (6.73+5,1);
  \coordinate (F2) at (6.73+5,1+1.5);
  \coordinate (F3) at (6.73+5,1-1.5);
  \coordinate (G) at (13,1);

  \draw[latex-, blue, line width=1.501pt] (D1) -- (E1);
  \draw[-latex, thick] (D1) -- (F1);
  \draw[-latex, thick] (D1) -- (F2);
  \draw[-latex, blue, line width=1.501pt] (D1) -- (F3);
  \draw[-latex, thick] (E1) -- (F1);
  \draw[-latex, blue, line width=1.501pt] (E1) -- (F2);
  \draw[-latex, thick] (E1) -- (F3);
  \draw[-latex, blue, line width=1.501pt] (F2) -- (G);
  \draw[-latex, blue, line width=1.501pt] (F3) -- (G);

  \foreach \i in {A,B,C}
   \draw[fill=white, line width=0.701pt] (\i) circle (3pt);

   \foreach \i in {A1,B1,C1,C2,C3,D1,E1,F1,F2,F3,G}
   \draw[fill=white, line width=0.701pt] (\i) circle (3pt);

    \coordinate (mid1) at (3,1.00001);   
    \coordinate (mid2) at (3.8,1.00001); 
    \draw[-Implies, double equal sign distance, thick] (mid1) -- (mid2);
    \coordinate (mid3) at (3+5,1.00001);   
    \coordinate (mid4) at (3.8+5,1.00001); 
    \draw[-Implies, double equal sign distance, thick] (mid3) -- (mid4);
    \node[left] at (B1) {$u$};
    \node[left] at (A1) {$v$};
    \node[left] at (D1) {$v$};
    \node[left] at (E1) {$u$};
    \node[right] at (C2) {$x_3$};
    \node[right] at (C3) {$x_1$};
    \node[right] at (C1) {$x_2$};
    \node[right] at (F2) {$x_3$};
    \node[right] at (F3) {$x_1$};
    \node[right] at (G) {$y$};
\end{tikzpicture}
}
\caption{Find $C''_5$ starting from an arbitrary triangle.}
\label{fig:C''}
\end{figure}

\noindent\textbf{(ii) Existence of $C'_5$.}

From part (i), we know $D$ contains a $C_5''$. Let its vertices be $\{v_1, \dots, v_5\}$ with arcs forming paths $v_1 \to v_2 \to v_3 \to v_4$ and $v_1 \to v_5 \to v_4$.
If $|N^+(v_1) \cap N^+(v_4)| \ge 3$, a $C_5'$ is easily found.
Thus we have $|N^+(v_1) \cap N^+(v_4)| \le 2$.
Define disjoint sets
\[
X = N^+(v_4) \setminus N^+(v_1), \quad Y = N^+(v_1) \setminus N^+(v_4), \quad Z = V(D) \setminus (X \cup Y).
\]
We have $|X|, |Y| \ge (1/3+\varepsilon)n - 2 > n/3$ and $|Z| \le (1/3 - 2\varepsilon)n + 2 < (1/3 - \varepsilon)n$.

If $\delta^+(D[X]) \ge 4$, then there is a directed path $x_1 \to x_2 \to x_3 \to x_4$ within $X$. 
This yields $C_5'$ composed of directed paths $v_4 \to x_1 \to x_2 \to x_3 \to x_4$ and $v_4\to x_4$. Otherwise, there exists $x \in X$ with $d^+(x,X) < 4$. The minimum out-degree implies:
\[
d^+(x,Y) \ge \delta^+(D) - d^+(x,X) - |Z| > (1/3+\varepsilon)n - 4 - (1/3 - \varepsilon)n - 2 > 0.
\]
Select $y \in N^+(x) \cap Y$. Then the vertices $\{v_1,v_5,v_4,x, y\}$ form a copy of $C_5'$, with directed paths $v_1\to v_5\to v_4\to x\to y$ and $v_1 \to y$ (see \cref{fig:C'}).

\begin{figure}[H]
\centering
\resizebox{0.27\linewidth}{!}{%
\begin{tikzpicture}[scale=1] 


  \coordinate (D1) at (5+5,0);
  \coordinate (E1) at (5+5,2);
  \coordinate (F2) at (6.73+5,1+1.5);
  \coordinate (F3) at (6.73+5,1-1.5);
  \coordinate (G) at (13,1);
  \coordinate (y) at (14.5,2.2);
  \coordinate (y') at (15,2.2);
  \coordinate (x) at (14.5,-0.2);
  \coordinate (x') at (15,-0.2);
  \coordinate (y1) at (14.2,2.8);
  \coordinate (y2) at (14.15,1.6);
  \coordinate (x1) at (14.2,-0.8);
  \coordinate (x2) at (14.15,0.4);

  \draw[-latex, red, line width=1.501pt] (E1) -- (D1);
  \draw[-latex, red, line width=1.501pt] (D1) -- (F3);
  \draw[-latex, red, line width=1.501pt] (F2) -- (E1);
  \draw[-latex, blue, line width=1.901pt] (F2) -- (G);
  \draw[-latex, blue, line width=1.901pt] (G) -- (F3);
  \draw[-latex, blue, line width=1.901pt] (F2) -- (y);
  \draw[-latex, blue, line width=1.901pt] (F3) -- (x);
  \draw[-latex, blue, line width=1.901pt] (x) -- (y);
  \draw[-latex, line width=1.001pt] (F2) -- (y1);
  \draw[-latex, line width=1.001pt] (F2) -- (y2);
  \draw[-latex, line width=1.001pt] (F3) -- (x1);
  \draw[-latex, line width=1.001pt] (F3) -- (x2);

  \draw[line width=1.201pt] (14.5, 2.2) ellipse (0.5cm and 0.8cm);
  \draw[line width=1.201pt] (14.5, -0.2) ellipse (0.5cm and 0.8cm);


   \foreach \i in {D1,E1,F2,F3,G,x,y}
   \draw[fill=white, line width=0.701pt] (\i) circle (3pt);
   \node[right] at (x') {\large $X$};
   \node[right] at (y') {\large $Y$};
   \node[below] at (x) {$x$};
   \node[above] at (y) {$y$};
   \node[above] at (F2) {$v_1$};
   \node[below] at (F3) {$v_4$};
   \node[right] at (G) {$v_5$};
   \node[left] at (D1) {$v_3$};
   \node[left] at (E1) {$v_2$};
\end{tikzpicture}
}
\caption{Find $C'_5$ starting from a $C''_5$.}
\label{fig:C'}
\end{figure}
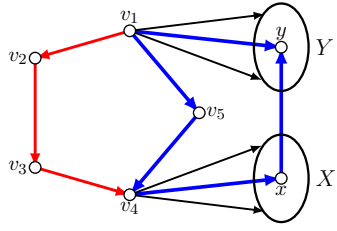

\noindent\textbf{(iii) Existence of $C'''_5$.}

We proceed by contradiction, assuming $D$ is $C_5'''$-free. First, there must exist an arc $(u_1, u_2)$ with $|N^+(u_1) \cap N^+(u_2)|\le 2$. 
Otherwise, the two endpoints of every arc have at least three common out-neighbors, implying a copy of $C_5'''$ (see \cref{fig:C'''}).

\begin{figure}[ht]
\centering
\resizebox{0.6\linewidth}{!}{%
\begin{tikzpicture}[scale=1] 
  \coordinate (A) at (0,0); 
  \coordinate (B) at (2,0); 
  \coordinate (C) at (2,2); 
  \coordinate (D) at (0,2); 
  \coordinate (D') at (0-5,2); 
  \coordinate (B') at (2-5,0); 
  \draw[dashed, thick] (B') -- (D');

  \coordinate (A1) at (0+5,0); 
  \coordinate (B1) at (2+5,0); 
  \coordinate (C1) at (2+5,2); 
  \coordinate (D1) at (0+5,2); 
  \coordinate (E1) at (8.5,1); 

  \draw[-latex, thick] (B) -- (A);
  \draw[-latex, thick] (D) -- (A);
  \draw[-latex, thick] (B) -- (C);
  \draw[-latex, thick] (D) -- (C);

  \draw[-latex, blue, line width=1.501pt] (B1) -- (A1);
  \draw[-latex, blue, line width=1.501pt] (D1) -- (A1);
  \draw[-latex, thick] (B1) -- (C1);
  \draw[-latex, blue, line width=1.501pt] (D1) -- (C1);
  \draw[-latex, blue, line width=1.501pt] (B1) -- (E1);
  \draw[-latex, blue, line width=1.501pt] (C1) -- (E1);

  \foreach \i in {A,B,C,D,B',D'}
   \draw[fill=white, line width=0.701pt] (\i) circle (3pt);

   \foreach \i in {A1,B1,C1,D1,E1}
   \draw[fill=white, line width=0.701pt] (\i) circle (3pt);

    \coordinate (mid1) at (3,1.00001);   
    \coordinate (mid2) at (3.8,1.00001); 
    \draw[-Implies, double equal sign distance, thick] (mid1) -- (mid2);
    \coordinate (mid3) at (3-5,1.00001);   
    \coordinate (mid4) at (3.8-5,1.00001); 
    \draw[-Implies, double equal sign distance, thick] (mid3) -- (mid4);
\end{tikzpicture}
}
\caption{Find $C'''_5$.}
\label{fig:C'''}
\end{figure}

Define:
\[
U_1 = N^+(u_1)\setminus N^+[u_2], \quad U_2 = N^+(u_2)\setminus N^+[u_1], \quad U_3 = V(D) \setminus (U_1 \cup U_2 \cup \{u_1, u_2\}),
\]
where $N^+[u_i]=N^+(u_i)\cup \{u_i\}$.
Then $|U_1|, |U_2| \ge (1/3+\varepsilon)n-3$ and $|U_3| < (1/3 - 2\varepsilon)n$.

\begin{claim}\label{cl:0-epsilon}
    For any $w \in V(D) \setminus \{u_1, u_2\}$, exactly one of $N^+(w) \cap U_1$ and $N^+(w) \cap U_2$ is empty.
    Moreover, if $N^+(w) \cap U_i$ is non-empty, then $|N^+(w) \cap U_i|> 3\varepsilon n$.
\end{claim}

\begin{pf}
    If both intersections were non-empty, there would exist $x \in N^+(w) \cap U_1$ and $y \in N^+(w) \cap U_2$. 
    Then $u_1\to u_2\to y \leftarrow w \to x \leftarrow u_1$ would form a $C_5'''$, a contradiction. 
    Thus at most one intersection is non-empty. 
    If $|N^+(w) \cap U_i| > 0$, then $|N^+(w) \cap U_{3-i}| = 0$, and
    \[
    |N^+(w) \cap U_i| \ge \delta^+(D) - |U_3| > (1/3+\varepsilon)n - (1/3 - 2\varepsilon)n = 3\varepsilon n.
    \]
\end{pf}

\begin{claim}\label{cl:odd-even}
    For any edge $ww'$ in $G[D - \{u_1, u_2\}]$, if $|N^+(w) \cap U_i| = 0$ for some $i \in [2]$, then $|N^+(w') \cap U_{3-i}| = 0$.
\end{claim}

\begin{pf}
    Suppose $|N^+(w) \cap U_i| = 0$. 
    By Claim \ref{cl:0-epsilon}, $|N^+(w) \cap U_{3-i}| > 3\varepsilon n$. 
    If $|N^+(w') \cap U_{3-i}| > 0$, then choose distinct vertices $x$, $y$ such that $x\in N^+(w) \cap U_{3-i}$ and $y\in N^+(w') \cap U_{3-i}$.
    This yields a copy of $C_5'''$ with $w\to x$, $u_{3-i}\to x$, $u_{3-i}\to y$, $w'\to y$ and $w'\to w$ (or $w\to w'$), a contradiction. 
    Thus $|N^+(w') \cap U_{3-i}| = 0$.
\end{pf}

\begin{claim}\label{cl:connect-to-u}
   For each $i\in[2]$, if $wu_i$ is an edge in $G(D)$, then $|N^+(w) \cap U_i| = 0$.
\end{claim}

\begin{pf}
    By symmetry, we only need to prove the case where $i=1$.
    Let $wu_1$ be an edge in $G(D)$ and suppose to the contrary that $|N^+(w) \cap U_1| > 0$. 
    If there exists $w' \neq w$ with $|N^+(w') \cap U_1| > 0$, then by Claim \ref{cl:0-epsilon}, $|N^+(w) \cap U_{2}| =|N^+(w') \cap U_{2}|= 0$. 
    Hence, we have $|N^+(w) \cap N^+(w')| \ge 2(1/3+\varepsilon)n - (n - |U_{2}|) >2\varepsilon n> 2$, allowing the selection of distinct vertices $x \in N^+(w) \cap N^+(w')$ and $y\in N^+(w')\cap U_1$.
    This yields a copy of $C_5'''$ with $w\to x, w'\to x, w'\to y, u_1\to y$ and $u_1\to w$ (or $w\to u_1$), a contradiction.
    Thus $N^+(w') \cap U_1 = \emptyset$ for all $w' \in V(D) \setminus \{u_1, u_2, w\}$.
    
    Consider the subgraph $D_1 = D - \{u_1, w\} - U_1$. We have:
    \[
    |D_1| = n - |U_1| - 2 \le (2/3 - \varepsilon)n+1.
    \]
    The minimum out-degree in $D_1$ satisfies:
    \[
    \delta^+(D_1) \ge \delta^+(D) - 2 > (1/3+\varepsilon/2)n>(1/2+\varepsilon)|D_1|.
    \]
    By \cref{cor:cycles}, $D_1$ contains $C_5'''$, so does $D$, a contradiction.
\end{pf}

We claim that the underlying graph $G(D - \{u_1, u_2\})$ is bipartite. 
Suppose, for contradiction, it contains an odd cycle $C: v_1 v_2 \dots v_{2\ell+1} v_1$. 
By Claim \ref{cl:0-epsilon}, for each $v_j$, exactly one of $N^+(v_j) \cap U_1$ or $N^+(v_j) \cap U_2$ is empty.
 Assume without loss of generality that $|N^+(v_1) \cap U_1| = 0$. 
Applying Claims \ref{cl:odd-even} iteratively along $C$: 
$|N^+(v_1) \cap U_1| = 0\Rightarrow |N^+(v_2) \cap U_2|=0\Rightarrow \cdots\Rightarrow |N^+(v_{2\ell+1}) \cap U_1|=0$.
Applying Claim \ref{cl:odd-even} again to $v_1v_{2\ell+1}$ and the condition $|N^+(v_1) \cap U_1| = 0$ forces $|N^+(v_{2\ell+1}) \cap U_2| = 0$. 
This contradicts Claim \ref{cl:0-epsilon} that exactly one of $N^+(v_{2\ell+1}) \cap U_1$ and $N^+(v_{2\ell+1}) \cap U_2$ is empty. 
Thus, $G(D - \{u_1, u_2\})$ is bipartite.

Since $g_{\text{odd}}(D) = 3$, there must be a triangle involving at least one of $u_1$ or $u_2$. The possible triangles are:
\begin{enumerate}
    \item $u_1 u_2 w$ for some $w \in V(D) \setminus \{u_1, u_2\}$.
    \item $u_i w w'$ for some $i \in [2]$ and $w, w' \in V(D) \setminus \{u_1, u_2\}$.
\end{enumerate}

  Consider the triangle $u_1 u_2 w$. By Claim \ref{cl:connect-to-u}, since $wu_1$ and $wu_2$ are edges, we have $|N^+(w) \cap U_1| = 0$ and $|N^+(w) \cap U_2| = 0$. This contradicts Claim \ref{cl:0-epsilon}. 

    Consider the triangle $u_i w w'$. 
    By Claim \ref{cl:connect-to-u}, we have $|N^+(w) \cap U_i| =|N^+(w') \cap U_i| = 0$. 
    Since $ww'$ is an edge, Claim \ref{cl:odd-even} applied with $|N^+(w) \cap U_i| = 0$ implies $|N^+(w') \cap U_{3-i}| = 0$. 
    Thus $|N^+(w') \cap U_1| = |N^+(w') \cap U_2| = 0$, contradicts Claim \ref{cl:0-epsilon}.

All cases yield a contradiction. Therefore, $D$ must contain $C_5'''$.
\end{proof}

We are now ready to prove the upper bound of \cref{thm:C5}.

\begin{theorem}\label{thm:bipartite} 
  Let $\varepsilon>0$ and $H\in\{C_5',C_5'',C_5'''\}$.  If $D$ is an $H$-free digraph with $\delta^+(D)> (1/3+\varepsilon)|D|$, then $D$ is bipartite.
\end{theorem}

\begin{proof}[\textbf{Proof}]
Let $|D|=n$. We assume for contradiction that $D$ is not bipartite, meaning its underlying graph $G(D)$ contains an odd cycle.
The proof is divided into three cases based on the odd girth of $D$.

\begin{case}
    $g_{\text{odd}}(D)=3$.
\end{case}

It follows directly from \cref{lmm:og=3} that this case is impossible.

\begin{case}
    $g_{\text{odd}}(D)=5$.
\end{case}

In this case, the underlying graph $G(D)$ contains a 5-cycle but no triangles. Our strategy is to show that we can take any oriented 5-cycle and iteratively modify it by replacing vertices with their common out-neighbors, eventually ``morphing" it into one of the forbidden pentagon orientations.

\begin{claim}\label{cl:anti-path}
   Let $C$ be an oriented 5-cycle in $D$. For any three consecutive vertices $u,v,w$ on $C$, there is a vertex $v'\notin V(C)$ such that $u\rightarrow v'$ and $w\rightarrow v'$.
\end{claim}

\begin{pf}
    Let $u,v,w$ be three consecutive vertices on $C$.
    Since $g_{\text{odd}}(D)=5$, there are no triangles in $G(D)$.
    Consequently, we have $N^+(u)\cap N^+(v)=\varnothing$ and $N^+(v)\cap N^+(w)=\varnothing$.
    Thus $|N^+(u)\cap N^+(w)|\ge d^+(u)+d^+(w)-(n-d^+(v))\ge 3\varepsilon n>|C|$.
    That is, there is a vertex $v'\notin V(C)$ such that $u\rightarrow v'$ and $w\rightarrow v'$.
\end{pf}

Let $C$ be an oriented 5-cycle in $D$ with vertex set $S=\{v_1,v_2,v_3,v_4,v_5\}$.
By iteratively applying \cref{cl:anti-path}, we obtain vertices:
\begin{itemize}
    \item $v_5' \notin S$ with $v_1 \rightarrow v_5'$ and $v_4 \rightarrow v_5'$;
    \item $v_4' \notin S'$ with $v_3 \rightarrow v_4'$ and $v_5' \rightarrow v_4'$, where $S' = S \setminus \{v_5\} \cup \{v_5'\}$;
    \item $v_3' \notin S''$ with $v_2 \rightarrow v_3'$ and $v_4' \rightarrow v_3'$, where $S'' = S' \setminus \{v_4\} \cup \{v_4'\}$;
    \item $v_2' \notin S'''$ with $v_1 \rightarrow v_2'$ and $v_3' \rightarrow v_2'$, where $S''' = S'' \setminus \{v_3\} \cup \{v_3'\}$;
\end{itemize}
The subgraph induced by the resulting vertex set at each step contains an oriented 5-cycle.
Then the subgraph induced by $\{v_1,v_2',v_3',v_4',v_5'\}$ contains a $C_5'$.

\begin{figure}[ht]
\centering
\resizebox{0.8\linewidth}{!}{%
\begin{tikzpicture}[scale=1] 
  \coordinate (A) at (0,0); 
  \coordinate (B) at (2,0); 
  \coordinate (C) at (2.618,1.902); 
  \coordinate (D) at (1,3.078); 
  \coordinate (E) at (-0.618,1.902); 

  \draw[dashed] (A) -- (B) -- (C) -- (D) -- (E) -- cycle;

  \foreach \i in {A,B,C,D,E}
   \draw[fill=white, line width=0.701pt] (\i) circle (3pt);

  \coordinate (A1) at (0+5.0001,0); 
  \coordinate (B1) at (2+5.0001,0); 
  \coordinate (C1) at (2.618+5.0001,1.902); 
  \coordinate (D1) at (1+5.0001,3.078); 
  \coordinate (E1) at (-0.618+5.0001,1.902); 

  \draw[dashed] (A1) -- (B1) -- (C1) -- (D1) -- (E1) -- cycle;
  \draw[-latex, line width=1.501pt] (A1) -- (E1);
  \draw[-latex, line width=1.501pt] (D1) -- (E1);

  \foreach \i in {B1,C1,E1}
   \draw[fill=white, line width=0.701pt] (\i) circle (3pt);
   \foreach \i in {A1,D1}
   \draw[fill=white, draw=blue, line width=1.301pt] (\i) circle (3pt);

  \coordinate (A2) at (0+10.0001,0); 
  \coordinate (B2) at (2+10.0001,0); 
  \coordinate (C2) at (2.618+10.0001,1.902); 
  \coordinate (D2) at (1+10.0001,3.078); 
  \coordinate (E2) at (-0.618+10.0001,1.902); 

  \draw[dashed] (A2) -- (B2) -- (C2) -- (D2) -- (E2) -- cycle;
  \draw[-latex, line width=1.501pt] (A2) -- (E2);
  \draw[-latex, line width=1.501pt] (E2) -- (D2);
  \draw[-latex, line width=1.501pt] (C2) -- (D2);

  \foreach \i in {A2,B2,C2,D2,E2}
   \draw[fill=white, line width=0.701pt] (\i) circle (3pt);
   \foreach \i in {C2,E2}
   \draw[fill=white, draw=blue, line width=1.301pt] (\i) circle (3pt);

  \coordinate (A3) at (0+15.0001,0); 
  \coordinate (B3) at (2+15.0001,0); 
  \coordinate (C3) at (2.618+15.0001,1.902); 
  \coordinate (D3) at (1+15.0001,3.078); 
  \coordinate (E3) at (-0.618+15.0001,1.902); 

  \draw[dashed] (A3) -- (B3) -- (C3) -- (D3) -- (E3) -- cycle;
  \draw[-latex, line width=1.501pt] (A3) -- (E3);
  \draw[-latex, line width=1.501pt] (E3) -- (D3);
  \draw[-latex, line width=1.501pt] (D3) -- (C3);
  \draw[-latex, line width=1.501pt] (B3) -- (C3);

  \foreach \i in {A3,B3,C3,D3,E3}
   \draw[fill=white, line width=0.701pt] (\i) circle (3pt);
   \foreach \i in {B3,D3}
   \draw[fill=white, draw=blue, line width=1.301pt] (\i) circle (3pt);

  \coordinate (A4) at (0+20.0001,0); 
  \coordinate (B4) at (2+20.0001,0); 
  \coordinate (C4) at (2.618+20.0001,1.902); 
  \coordinate (D4) at (1+20.0001,3.078); 
  \coordinate (E4) at (-0.618+20.0001,1.902); 

  \draw[-latex, line width=1.501pt] (A4) -- (E4);
  \draw[-latex, line width=1.501pt] (E4) -- (D4);
  \draw[-latex, line width=1.501pt] (D4) -- (C4);
  \draw[-latex, line width=1.501pt] (C4) -- (B4);
  \draw[-latex, line width=1.501pt] (A4) -- (B4);

  \foreach \i in {A4,B4,C4,D4,E4}
   \draw[fill=white, line width=0.701pt] (\i) circle (3pt);
   \foreach \i in {A4,C4}
   \draw[fill=white, draw=blue, line width=1.301pt] (\i) circle (3pt);

  \coordinate (mid1) at (3.2,1.50001);   
  \coordinate (mid2) at (4,1.50001);      
  \draw[-Implies, double equal sign distance, thick] (mid1) -- (mid2);
  \coordinate (mid3) at (3.2+5,1.50001);   
  \coordinate (mid4) at (4+5,1.50001);      
  \draw[-Implies, double equal sign distance, thick] (mid3) -- (mid4);
  \coordinate (mid5) at (3.2+5+5,1.50001);   
  \coordinate (mid6) at (4+5+5,1.50001);      
  \draw[-Implies, double equal sign distance, thick] (mid5) -- (mid6);
  \coordinate (mid7) at (3.2+5+5+5,1.50001);   
  \coordinate (mid8) at (4+5+5+5,1.50001);      
  \draw[-Implies, double equal sign distance, thick] (mid7) -- (mid8);
  
\end{tikzpicture}
}
\caption{Find $C_5'$ starting from an arbitrary 5-cycle.}
\end{figure}
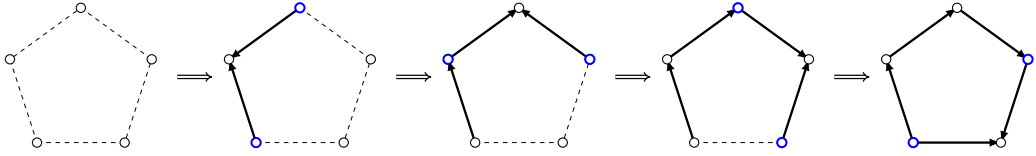

By iteratively applying \cref{cl:anti-path}, we can also find $C_5''$ and $C_5'''$ in $D$. 
The argument is analogous to the one used to obtain $C_5'$, so we omit the details and illustrate the construction in \cref{fig:C''and C'''}, where the third digraph contains $C_5'''$, and the fourth digraph contains $C_5''$.

\begin{figure}[ht]
\centering
\resizebox{0.6\linewidth}{!}{%
\begin{tikzpicture}[scale=1] 
  \coordinate (A) at (0,0); 
  \coordinate (B) at (2,0); 
  \coordinate (C) at (2.618,1.902); 
  \coordinate (D) at (1,3.078); 
  \coordinate (E) at (-0.618,1.902); 

  \draw[dashed] (A) -- (B) -- (C) -- (D) -- (E) -- cycle;

  \foreach \i in {A,B,C,D,E}
   \draw[fill=white, line width=0.701pt] (\i) circle (3pt);

  \coordinate (A1) at (0+5.0001,0); 
  \coordinate (B1) at (2+5.0001,0); 
  \coordinate (C1) at (2.618+5.0001,1.902); 
  \coordinate (D1) at (1+5.0001,3.078); 
  \coordinate (E1) at (-0.618+5.0001,1.902); 

  \draw[dashed] (A1) -- (B1) -- (C1) -- (D1) -- (E1) -- cycle;
  \draw[-latex, line width=1.501pt] (A1) -- (E1);
  \draw[-latex, line width=1.501pt] (D1) -- (E1);

  \foreach \i in {B1,C1,E1}
   \draw[fill=white, line width=0.701pt] (\i) circle (3pt);
   \foreach \i in {A1,D1}
   \draw[fill=white, draw=blue, line width=1.301pt] (\i) circle (3pt);

  \coordinate (A2) at (0+10.0001,0); 
  \coordinate (B2) at (2+10.0001,0); 
  \coordinate (C2) at (2.618+10.0001,1.902); 
  \coordinate (D2) at (1+10.0001,3.078); 
  \coordinate (E2) at (-0.618+10.0001,1.902); 

  \draw[dashed] (A2) -- (B2) -- (C2) -- (D2) -- (E2) -- cycle;
  \draw[-latex, line width=1.501pt] (A2) -- (E2);
  \draw[-latex, line width=1.501pt] (D2) -- (E2);
  \draw[-latex, line width=1.501pt] (D2) -- (C2);
  \draw[-latex, line width=1.501pt] (B2) -- (C2);

  \foreach \i in {A2,B2,C2,D2,E2}
   \draw[fill=white, line width=0.701pt] (\i) circle (3pt);
   \foreach \i in {B2,D2}
   \draw[fill=white, draw=blue, line width=1.301pt] (\i) circle (3pt);

  \coordinate (A3) at (0+15.0001,0); 
  \coordinate (B3) at (2+15.0001,0); 
  \coordinate (C3) at (2.618+15.0001,1.902); 
  \coordinate (D3) at (1+15.0001,3.078); 
  \coordinate (E3) at (-0.618+15.0001,1.902); 

  \draw[dashed] (A3) -- (B3) -- (C3) -- (D3) -- (E3) -- cycle;
  \draw[-latex, line width=1.501pt] (A3) -- (E3);
  \draw[-latex, line width=1.501pt] (E3) -- (D3);
  \draw[-latex, line width=1.501pt] (C3) -- (D3);
  \draw[-latex, line width=1.501pt] (B3) -- (C3);

  \foreach \i in {A3,B3,C3,D3,E3}
   \draw[fill=white, line width=0.701pt] (\i) circle (3pt);
   \foreach \i in {C3,E3}
   \draw[fill=white, draw=blue, line width=1.301pt] (\i) circle (3pt);

  \coordinate (mid1) at (3.2,1.50001);   
  \coordinate (mid2) at (4,1.50001);      
  \draw[-Implies, double equal sign distance, thick] (mid1) -- (mid2);
  \coordinate (mid3) at (3.2+5,1.50001);   
  \coordinate (mid4) at (4+5,1.50001);      
  \draw[-Implies, double equal sign distance, thick] (mid3) -- (mid4);
  \coordinate (mid5) at (3.2+5+5,1.50001);   
  \coordinate (mid6) at (4+5+5,1.50001);      
  \draw[-Implies, double equal sign distance, thick] (mid5) -- (mid6);
\end{tikzpicture}
}
\caption{Find $C_5''$ and $C_5'''$ starting from an arbitrary 5-cycle.}
\label{fig:C''and C'''}
\end{figure}
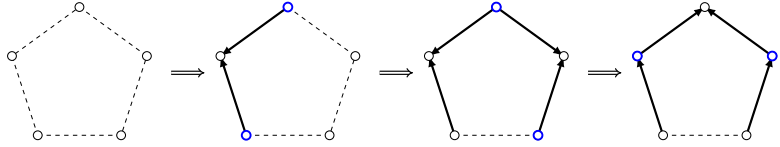

\begin{case}
    $g_{\text{odd}}(D)\ge 7$.
\end{case}

    Let $C$ be a minimum odd cycle in $G(D)$.
    Then $C$ is an induced cycle and $|C|\ge 7$.
    Moreover, $d^+(V(C),v)\le 2$ for every $v\in V(D)$; otherwise, there exists an odd cycle of length smaller than $|C|$.
    By double counting the arcs from $V(C)$ to $V(D)\setminus V(C)$, we have
\begin{align*}
   \sum_{u\in V(C)}(d^+(u)-2)\le \sum_{u\in V(C)}d^+(u,V(D)\setminus V(C))=\sum_{v\in V(D)\setminus V(C)} d^+(V(C),v)\le 2(n-|C|).
\end{align*}
  By the pigeonhole principle, there is a vertex $u\in V(C)$ such that $d^+(u)-2\le \frac{2(n-|C|)}{|C|}$.
  This implies that $d^+(u)\le \frac{2n}{|C|}\le \frac{2n}{7}<\frac{n}{3}$, a contradiction.

  Since all possible cases for the odd girth lead to a contradiction, $D$ must be bipartite.
\end{proof}

\begin{proof}[\textbf{Proof of Theorem \ref{thm:C5}}]
  The lower bound $\delta^+_\chi(H,2)\ge 1/3$ for each $H\in\{C_5',C_5'',C_5'''\}$ is given by \cref{prop:3.2} and  \cref{prop:3.3}.
  The upper bound $\delta^+_\chi(H,2)\le 1/3$ follows from \cref{thm:bipartite}.
  This completes the proof.
\end{proof}

\section{An application of the directed version Andr\'asfai-Erd\H{o}s-S\'os theorem}\label{sec5}

In this section, we provide an application of the directed version Andr\'asfai-Erd\H{o}s-S\'os theorem. 

First we introduce the famous regularity lemma.
Let $D$ be a digraph and $G=G(D)$ be its underlying graph.
Let $(A, B)$ be a pair of subsets of vertices of $G$. 
Let $e(A, B)$ denote the number of edges with one endpoint in $A$ and the other in $B$. 
The \emph{density} of the pair $(A, B)$ is defined as $d(A, B)=\frac{e(A, B)}{|A||B|}$, in the undirected case. 
For any $\varepsilon>0$, the pair $(A, B)$ is said to be \emph{$\varepsilon$-regular} if $|d(A, B)-d(X, Y)|<\varepsilon$ for every $X \subseteq A$ and $Y \subseteq B$ with $|X| \geq \varepsilon|A|$ and $|Y| \geq \varepsilon|B|$. Moreover, given $0 < d < 1$, we say that $(A, B)$ is \emph{$(\varepsilon, d)$-regular} if it is $\varepsilon$-regular and has density at least $d$.
A partition $V_0 \cup V_1 \cup \cdots \cup V_k$ of $V(G)$ is called an \emph{$\varepsilon$-regular partition} if $|V_0| \leq \varepsilon n,|V_1|=\cdots=|V_k|$, and all but at most $\varepsilon k^2$ of the pairs $(V_i, V_j)$, where $1\le i,j\le k,$ are $\varepsilon$-regular. 
The vertex sets $V_1, \ldots, V_k$ are called \textit{clusters}. 
Given clusters $V_1, \ldots, V_k$ and a digraph $D$, the \textit{$(\varepsilon, d)$-reduced digraph} $R$ is a digraph on the vertex set $[k]$.
There is an arc from $i$ to $j$ if and only if the pair $(V_i,V_j)$ is $(\varepsilon,d)$-regular.

The Diregularity lemma is a version of the Regularity lemma for digraphs due to Alon and Shapira \cite{alon2003testing}.
We will use the following minimum degree form of the diregularity lemma.

\begin{lemma}[Degree form of the Diregularity Lemma, see \cite{taylor2014regularity} and \cite{keevash2009exact}]\label{Diregularity}
    Let $0<2\varepsilon<d<\delta<1$, and let $k_0 \in \mathbb{N}$. There exists a constant $k_1=k_1(k_0, \varepsilon, \delta, d)$ such that the following holds. Every digraph $D$ on $n>k_1$ vertices, with minimum out-degree $\delta^+(D) \geq \delta n$, has an $(\varepsilon, d)$-reduced digraph $R$ on $k$ vertices, with $k_0 \leq k \leq k_1$ and $\delta^+(R) \geq(\delta-d-3\varepsilon) k$. 
\end{lemma}

\begin{proof}[\textbf{Proof of Theorem \ref{thm:app-of-AES}}]
Fix  $1\gg\gamma\gg d\gg \varepsilon>0$.
Applying minimum degree form regularity lemma (\cref{Diregularity}) to $D$, we obtain a partition $V(D)=V_0\cup V_1\cup \cdots \cup V_k$ of $D$, together with an $(\varepsilon,d)$-reduced graph $R$ on the vertex set $V(R)=[k]$, such that $\delta^+(R)\geq(\frac{3r-7}{3r-4}+\frac{\gamma}{2})k$.    

Since the original digraph $D$ is $T_{r}[t]$-free, a standard result known as the Embedding Lemma implies that for a sufficiently small $\varepsilon$, the reduced digraph $R$ must be $T_{r}$-free.
By \cref{thm:directed-AES}, $R$ is $(r-1)$-colorable.

Suppose that $I$ is an independent set in $R$ and $V_I=\cup_{i\in I} V_i$.
The arcs of $D[V_I]$ can be categorized into the following three types:
\begin{itemize}
    \item arcs with both endpoints in $V_i$ for some $i \in I$. There are at most $2k\binom{n / k}{2}$ such arcs;
    \item arcs between distinct clusters $V_i$ and $V_j$, for some distinct $i, j \in I$ for which $(V_i, V_j)$ is not $\varepsilon$-regular. There are at most $2\varepsilon k^2 \frac{n^2}{k^2}$ such arcs;
    \item arcs between distinct clusters $V_i$ and $V_j$, for some distinct $i, j \in I$ for which $(V_i, V_j)$ is $\varepsilon$-regular but $d(V_i, V_j)<d$. There are at most $2d\binom{k}{2} \frac{n^2}{k^2}$ such arcs.
\end{itemize}

Summing up the arcs from these three types yields $|A(D[V_I])|\le (4\varepsilon+d)n^2$.
Since $V(R)$ can be divided into $r-1$ independent sets, this implies that to make $D$ $(r-1)$-partite it suffices to delete at most $(r-1)(4\varepsilon+d)n^2=o(n^2)$ arcs.
\end{proof}

\paragraph{Remark.} 
The $T_r[t]$-free condition in Theorem \ref{thm:app-of-AES} cannot be replaced by forbidding an arbitrary digraph that also has chromatic number $r$. 
Consider the following construction for $r=3$, which is a modification of \cref{fig:cons} (a). 
Let $D$ be a digraph with a vertex partition $A \cup B \cup C$, where $|A|=(1/5-2\varepsilon)n$ and $|B|=|C|=(2/5+\varepsilon)n$. 
The arc set consists of all arcs from $A$ to $B\cup C$, and all arcs between $B$ and $C$ (i.e., from $B$ to $C$ and from $C$ to $B$).
This digraph is $\overrightarrow{C_3}$-free with minimum out-degree $\delta^{+}(D) = (2/5+\varepsilon)n$, which meets the theorem's threshold. 
However, the digraph is not close to being bipartite, as one must delete at least $\Theta(n^2)$ arcs to make it so.

\section{Concluding Remarks}\label{sec6}

In contrast to the chromatic profile defined in \eqref{eq:profile}, which requires the chromatic number to be bounded by a specific integer $k$, a related notion is the chromatic threshold.
The \emph{directed chromatic threshold} of a digraph $H$ is defined as:
\begin{align*}
\delta^+_\chi(H):=
& \inf \{d: \exists~C=C(H, d)\text{ such that if $D$ is a digraph on $n$ vertices}, \\
& \text{with $\delta^+(D) \geq d n$ and $H \not \subseteq D$, then $\chi(D) \leq C$}\}.
\end{align*}

It is instructive to compare our results on the 2-coloring profile $\delta_{\chi}^{+}(H,2)$ with the known chromatic thresholds for the four pentagon orientations discussed.  Ding et al. \cite{Directed} established that $\delta_{\chi}^{+}(\overrightarrow{C_{5}})=1/2$ and $\delta_{\chi}^{+}(C_{5}')=1/4$.  In contrast, Koerts, Moore, and Spirkl \cite{koerts2025orientations} showed that the threshold is zero for the other two cases: $\delta_{\chi}^{+}(C_{5}'')=\delta_{\chi}^{+}(C_{5}''')=0$.
Notably, the values and even the relative ordering of these two parameters can differ significantly. For instance, we proved $\delta_{\chi}^{+}(C_{5}',2) = \delta_{\chi}^{+}(C_{5}'',2) = 1/3$, whereas their chromatic thresholds are $1/4$ and $0$, respectively. This complex relationship motivates further study, and we pose the following open questions:

\begin{question}
  Let $H_1$ and $H_2$ be two orientations of a graph $H$.  Does the relation $\delta_{\chi}^{+}(H_{1})>\delta_{\chi}^{+}(H_{2})$ imply that $\delta_{\chi}^{+}(H_{1},k)\ge\delta_{\chi}^{+}(H_{2},k)$ for every integer $k\ge 2$?
\end{question}

It is observed that \cref{cons3} and \cref{cons4} provide a lower bound of 1/3 for certain oriented odd cycles.

\begin{question}
   Determine the chromatic profile for every orientation of an arbitrary odd cycle $C_{2\ell+1}$.
\end{question}


We note that our constructions for the lower bounds contain vertices with zero in-degree. It is therefore a natural problem to study the chromatic profile where the minimum out-degree is replaced by the minimum semi-degree, a setting in which our constructions are no longer extremal.

Finally, Alon and Sudakov \cite{alon2006h} proved that if $G$ is a $K_{r}[t]$-free graph of order $n$ with minimum degree $\delta(G) \geq(\frac{3r-7}{3r-4}+\varepsilon) n$, then one can delete at most $O(n^{2-1 /(4 (r-1)^{2 / 3} t)})$ edges to make $G$ $(r-1)$-colorable.
This invites the question of whether a similar quantitative improvement can be made to our stability result.


\begin{question}
    Let $n$ be sufficiently large and $D$ be a $T_{r}[t]$-free digraph on $n$ vertices with minimum out-degree $\delta^{+}(D) \geq(\frac{3 r-7}{3 r-4}+\varepsilon) n$. Does there exist a constant $\alpha:=\alpha(r,t)>0$, such that one can delete $O(n^{2-\alpha})$ arcs to make $D$ an $(r-1)$-partite digraph?
\end{question}

\section*{Acknowledgments}
The author thanks Zilong Yan for fruitful and enlightening discussions in the early stages of this project and is grateful to Yuefang Sun for carefully reading the manuscript and Hong Liu for valuable writing suggestions.

\bibliographystyle{abbrv}
\bibliography{references.bib}

\end{document}